\documentclass[draft]{amsart}
\usepackage{amssymb} 
\newtheorem{theorem}{Theorem}[section]

\newtheorem{proposition}[theorem]{Proposition}
\newtheorem{corollary}[theorem]{Corollary}

\theoremstyle{definition}
\newtheorem{definition}[theorem]{Definition}

\newtheorem{remark}[theorem]{Remark}
\theoremstyle{approach}

\numberwithin{equation}{section}

\begin{document}
\setcounter{page}{1}
\title[locally bounded approximate diagonal  modulo an ideal]{locally bounded approximate diagonal modulo an ideal  of Fr$\acute{E}$chet algebras}
\author[S. Rahnama and A. Rejali]{S. Rahnama and A. Rejali}



\begin{abstract}
For a Banach algebra $\mathcal A$ and a closed ideal $I$, 
the notion of bounded approximate diagonal modulo $I$ has been 
studied and investigated. In this paper we define the notion of 
locally bounded approximate diagonal modulo an ideal $I$ for a
Fr$\acute{e}$chet  algebra  $\mathcal A$ and obtain the relation between 
amenability modulo an ideal $I$ and the existence of locally
 bounded approximate diagonals modulo $I$.
\end{abstract}

\maketitle \setcounter{section}{-1}

\section{\bf Introduction}
Amenability of  Fr$\acute{e}$chet algebras   was first
introduced by Helemskii in \cite{Hel4} and \cite{Hel5}. Also in \cite{Pir} 
 Pirkovskii  studied this notion and obtained some results 
 about amenability of  Fr$\acute{e}$chet algebras. 
In \cite{Rahimi Amini} Amini and Rahimi introduced the notion of amenability modulo an
 ideal of Banach algebras. 
 Also they defined the notions of bounded approximate identities
 and bounded approximate diagonals modulo an ideal for a Banach algebra.
 They showed if  $\mathcal A$ is a Banach algebra and $I$ 
 is a closed ideal  of $\mathcal A$, $\mathcal A$ is amenable modulo $I$
 if and only if $\mathcal A$ has a bounded approximate diagonal modulo $I$.
Furthermore they prove that if $\mathcal A$ is amenable modulo 
$I$, it has a bounded approximate identity modulo $I$.

In this paper we generalize definition of bounded approximate diagonal 
modulo an ideal in Banach algebras for  Fr$\acute{e}$chet algebras.
The paper is organized as follows.
 In section $1$ we present some preliminaries and definitions related 
 to locally convex spaces and  Fr$\acute{e}$chet algebras. In section $2$
 we define the notion of locally bounded approximate diagonal modulo 
 an ideal for a  Fr$\acute{e}$chet algebra. 
Mainly  we show that amenability modulo an ideal for a Fr$\acute{e}$chet
algebra is equivalent to the existence of a locally bounded approximate 
diagonal modulo an ideal for this Fr$\acute{e}$chet
algebra.
 Finaly we prove that $\mathcal A$ is amenable modulo 
an ideal $I$ 
 if and only if ${\mathcal A}_e$, the unitization of
$\mathcal A$, is  amenable modulo $I$.
\section{\bf Preliminaries}

In this section, we mention some  definitions and properties about
 locally convex spaces and also Fr$\acute{e}$chet
algebras, which will be used throughout the paper. See 
\cite{MA} and \cite{ME} for more informations.

A   locally convex space is a topological vector space, whose topology 
is generated by translations of balanced, absorbent, convex sets.
Equivalenty a locally convex space is a topological vector space whose
topology is defined by a  fundamental system of seminorms. 
 We denote by $(E,p_{\alpha})$, a
locally convex space $E$ with a fundamental system of seminorms
$(p_{\alpha})_{\alpha}$.
All locally convex spaces assume to be Hausdorff in this paper.

A  topological algebra  ${\mathcal A}$ is an algebra which 
is a topological vector space and the   multiplication
$${\mathcal A}\times{\mathcal A}\longrightarrow {\mathcal A}, \;\;\;\; (a,b)\mapsto ab$$
is a separately continuous mapping; see \cite{Pir}. An important
 class of topological algebra is the class of Fr$\acute{e}$chet algebras. A
Fr$\acute{e}$chet algebra, denoted by $(\mathcal A, p_n)$, is a
complete topological algebra, whose topology is given by the
countable family of increasing submultiplicative seminorms. 
Also every closed subalgebra of a
Fr$\acute{e}$chet algebra is clearly a Fr$\acute{e}$chet algebra.

For  a  Fr$\acute{e}$chet algebra $(\mathcal A,p_n)$,  a
locally convex $\mathcal A$-bimodule is a locally convex
topological vector space $X$  with an
$\mathcal A$-bimodule structure  such that the corresponding mappings are
separately continuous. 
 Let $(\mathcal A,p_n)$ be a Fr$\acute{e}$chet algebra and $X$ be a 
locally convex  $\mathcal A$-bimodule.  Following \cite{Hel3}, a continuous  
derivation of $\mathcal A$ into $X$ is a continuous mapping 
$D$ from $\mathcal A$ into $X$ such that 
$$D(ab)=a.D(b)+D(a).b,$$
for all $a,b\in \mathcal A$. Furthermore for each 
$x\in X$ the mapping $\delta_x:\mathcal A\to X$ defined by
$$\delta_x(a)=a.x-x.a\;\;\;\;\;(a\in\mathcal A),$$
is a continuous derivation and is called the inner derivation associated with 
$x$.

We recall definition of an inverse limit from \cite{F}. Let we are given 
a family $(E_{\alpha})_{{\alpha}\in \Lambda}$ of algebras, 
where  $\Lambda$ is  a directed  set. Suppose we are given a family 
$f_{\alpha\beta}$ of homomorphisms such that $f_{\alpha\beta}:E_{\beta}\to E_{\alpha}$
for any ${\alpha,\beta\in \Lambda}$, with $\alpha\leq \beta$. A
projective system of algebras is  a family $\{(E_{\alpha}, f_{\alpha\beta})\}$
as above which in addition to the above relation, satisfies also the following condition 
$$f_{\alpha\gamma}=f_{\alpha\beta}\circ f_{\beta\gamma},\;\;\; (\alpha,\beta,\gamma\in\Lambda,\alpha\leq\beta\leq\gamma). $$
Now consider the cartesian product algebra  $F=\prod_{\alpha\in\Lambda}E_{\alpha}$, as well as the following subset of $F$,
$$E=\{x=(x_{\alpha})\in F:x_{\alpha}=f_{\alpha\beta}(x_{\beta}),\;\; \alpha\leq \beta\}. $$
Then $E$ is the projective (or inverse) limit of the given projective system
 $\{(E_{\alpha}, f_{\alpha\beta})\}$ and is denoted by $E= \underleftarrow{\lim}\{(E_{\alpha}, f_{\alpha\beta})\}$
or simply $E=\underleftarrow{\lim}E_{\alpha}$.

Now let $(\mathcal A, p_{\lambda})$ be a locally convex algebra.
For each ${\lambda\in \Lambda}$ consider the ideal $N_{\lambda}= \ker p_{\lambda}$ and
 the respective normed algebra $\frac{\mathcal A}{N_{\lambda}}$. 
Suppose that $\varphi_{\lambda}:\mathcal A\to\mathcal A_{\lambda},\varphi_{\lambda}(x)=x_{\lambda}=x+{N_{\lambda}} $
be the corresponding quotient map, which is clearly a continuous surjective 
homomorphism. Now for any $\lambda, \gamma\in\Lambda$ with $\lambda\leq \gamma$, one 
has $N_{\lambda}\subseteq N_{\gamma}$. So that the linking maps
$$\varphi_{\lambda\gamma}:\frac{\mathcal A}{N_{\gamma}}\to\frac{\mathcal A}{N_{\lambda}},\;\;\;\varphi_{\lambda\gamma}(x+N_{\gamma})=x+{N_{\lambda}}, $$
are well defined continuous surjective homomorphisms such that
$\varphi_{\lambda\gamma}\circ \varphi_{\gamma}=\varphi_{\lambda}$. Hence $\varphi_{\lambda\gamma}$'s
are uniquely extended to continuous homomorphisms between
the Banach algebras $\mathcal A_{\gamma}$ and $\mathcal A_{\lambda}$, we retain the symbol 
$\varphi_{\lambda\gamma}$ for the extentions too, which  
$\mathcal A_{\gamma}$ is the compeletion of $\frac{\mathcal A}{N_{\gamma}}$.
The families 
$(\frac{\mathcal A}{N_{\lambda}},\varphi_{\lambda\gamma} )$ [respectively, $(\mathcal A_{\lambda},\varphi_{\lambda\gamma} )$],
form inverse system of 
normed,[ respectively, Banach] algebras. Their corresponding inverse limits denoted 
by $\underleftarrow{\lim}\frac{\mathcal A}{N_{\lambda}}$ and $\underleftarrow{\lim}\mathcal A_{\lambda}$.
Moreover in the case when the initial algebra $\mathcal A$ is complete, one has 
$$\mathcal A=\underleftarrow{\lim}\frac{\mathcal A}{N_{\lambda}}  =\underleftarrow{\lim}\mathcal A_{\lambda},$$
up to topological isomorphisms.
Let $\mathcal A$ be a Fr$\acute{e}$chet algebra
which its topology defined by the family of increasing
 submultiplicative seminorms $(p_n)_{n\in\mathbb{N}}$.
For each $n\in\mathbb{N}$ let $\varphi_n:\mathcal A\to \frac{\mathcal A}{\ker p_n}$
be the quotient map. Then $\frac{\mathcal A}{\ker p_n}$ is naturally a normed
algebra, normed by setting $\|\varphi_n(a)\|_n=p_n(a)$ for each $a\in\mathcal A$.
The compeletion  $(\mathcal A_n, \|.\|_n)$ 
is a Banach algebra. Henceforth we consider $\varphi_n$ as a mapping from $\mathcal A$ into
$\mathcal A_n$. We call it the canonical map. It is important to note that $\varphi_n(\mathcal A)$ is a dense 
subalgebra of $\mathcal A_n$ and  in general $\mathcal A_n\neq\varphi_n(\mathcal A)$.
Since $p_n\leq p_{n+1}$, there is a naturally induced, norm-decreasing
homomorphism $d_n:\mathcal A_{n+1}\to \mathcal A_n$ such that 
$d_n\circ \varphi_{n+1}=\varphi_{n}$ for each $n\in\mathbb{N}$. One can show  that 
$\mathcal A=\underleftarrow{\lim}(\mathcal A_n, d_n)$. 

The above is the Arense-Michael decomposition of $\mathcal A$, 
which expresses Fr$\acute{e}$chet algebra as an reduced inverse limit of Banach algebras. 

Now choose an Arens-Michael decomposition  $\mathcal A=\underleftarrow{\lim}\mathcal A_n$
and let $I$ be a closed ideal of $\mathcal A$.
Then it is easy to see that $I=\underleftarrow{\lim}\overline{I_n}$ is an 
Arens-Michael decomposition of $I$, where $\varphi_{n}:I\to I_n$
is the canonical map.(see \cite{Pir}).

According  to \cite{Pir} if $\mathcal A$ is a locally convex algebra
and $X$ is a left locally convex  $\mathcal A$-module, then  a continuous
seminorm $q$ on $X$  is m-compatible if there exists a continuous 
submultiplicative seminorm $p$ on  $\mathcal A$ such that 
$$q(a.x)\leq p(a)q(x),\;\;\;(a\in \mathcal A, x\in X).$$
Also by \cite[3.4]{Pir1} if $\mathcal A$ is a Fr$\acute{e}$chet algebra
and $X$ is a complete left  $\mathcal A$-module
with a jointly continuous left module action, then the topology
on $X$ can be determined by a directed family of m-compatible
seminorms.

\section{\bf Locally bounded approximate diagonal  modulo an ideal of a  Fr$\acute{e}$chet algebra }

Let $\mathcal A$ be a Banach algebra and $I$ be a closed ideal of $\mathcal A$. According to
\cite{Rahimi  Amini}, $\mathcal A$ is amenable modulo $I$, if for every Banach 
$\mathcal A$-bimodule $E$ such that $I.E=E.I=0$ and every derivation $D$ from $\mathcal A$
into $E^*$, there exists $\varphi\in E^*$ such that
$$D(a)=a.\varphi-\varphi.a,  \;\;\;\;(a\in\mathcal A\setminus I).$$
In \cite[Definition 2.1]{Rahnama} we introduced the  definition of amenability  modulo an ideal
for a Fr$\acute{e}$chet algebra. Also we  extended some 
results of \cite{Rahimi1}, for Fr$\acute{e}$chet algebras.
In fact from \cite{Rahnama}, we have the following definition.

\begin{definition}\label{D1}\rm
Let $(\mathcal A, p_{n})$ be a Fr$\acute{e}$chet algebra and $I$ be a closed ideal
of  $\mathcal A$. 
 $\mathcal A$  is called amenable modulo $I$,  if  for every Banach  
$\mathcal A$-bimodule $E$ such that $E.I=I.E=0$ each
continuous derivation from $\mathcal A$ into $E^*$ is
 inner on $\mathcal A\setminus I$.
\end{definition}

In this section we discuss about approximate diagonal modulo an ideal for a Fr$\acute{e}$chet algebra.

Let $(\mathcal A, p_n)$ be a  Fr$\acute{e}$chet algebra and  $I$ be a closed ideal of 
$\mathcal A$. 
By \cite[Lemma 22.9]{ME} $\frac{\mathcal A}{I}$ endowed with the quotient topology is
a  Fr$\acute{e}$chet space and the topology is defined by the seminorms
$$\hat{p}_n(a+I)=\inf\{p_n(a+b):\;\;\; b\in I\}.$$
Moreover the multiplication 
$$(a+I,b+I)\mapsto ab+I\;\;\;\;(a,b\in\mathcal A),$$
on $\frac{\mathcal A}{I}$ is continuous and $ (\frac{\mathcal A}{I},\hat{p}_n)$ is a  Fr$\acute{e}$chet algebra; see\cite[3.2.10]{Gold}.
Also we recall projective tensor product of Fr$\acute{e}$chet algebra 
$(\mathcal A, p_n)$ which has been introduced in \cite{Shi}.  It will be denoted by 
$(\mathcal A\widehat{\otimes}\mathcal A, r_n)$, where 
$$r_n(u)=\inf\{\sum_{i\in\mathbb{N}}p_n(a_i)p_n(b_i);\;\;\;u=\sum_{i\in\mathbb{N}}a_i\otimes b_i\},$$
for each $u\in \mathcal A\widehat{\otimes}\mathcal A$. By \cite{Shi} and \cite[Theorem 2]{SMI}
and also \cite[Theorem 45.1]{Treves}, $(\mathcal A\widehat{\otimes}\mathcal A, r_n)$
is again a  Fr$\acute{e}$chet algebra. Also  $(\mathcal A\widehat{\otimes}\mathcal A, r_n)$
is a Fr$ \acute{e}$chet $\mathcal A$-bimodule with the following module
actions 
$$a.(b\otimes c)=ab\otimes c \;\;\; and \;\;\; (b\otimes c).a= b\otimes ca \;\;\;(a,b,c\in\mathcal A).$$
Also the corresponding diagonal operator of 
$\frac{\mathcal A}{I}$ is defined by 
$$\pi_{\frac{\mathcal A}{I}}: \frac{\mathcal A}{I}\widehat{\otimes}  \frac{\mathcal A}{I}\to  \frac{\mathcal A}{I}, \;\;\;\;\; (a+I)\otimes(b+I)\mapsto ab+I,\;\;\;\;(a,b\in \mathcal A).$$
Clearly $\pi_{\frac{\mathcal A}{I}}$ is an $\mathcal A$-bimodule homomorphism.

A particularly important result related with amenability modulo an ideal 
of a Banach algebra is the existence of a bounded approximate diagonal
and virtual diagonal modulo an ideal for a Banach algebra.
In fact according to \cite[Definition 6]{Rahimi1}, if $\mathcal A$
is a Banach algebra and $I$ is a closed ideal of $\mathcal A$, 
a bounded net $(m_{\alpha})\subseteq \frac{\mathcal A}{I}\widehat{\otimes}  \frac{\mathcal A}{I}$
is an approximate diagonal modulo $I$ if 
$$\lim_{\alpha}(a.\pi_{\frac{\mathcal A}{I}}(m_{\alpha})-\tilde{a})=0,\;\;\;(a\in\mathcal A, \;\;\tilde{a}=a+I),$$
$$\lim_{\alpha}(a.m_{\alpha}-m_{\alpha}.a)=0\;\;\;\;(a\in \mathcal A\setminus  I).$$

Let us now consider the  Fr$\acute{e}$chet algebra $(\mathcal A, p_n)$. 
By \cite{Pir}, a bounded net $(m_{\alpha})\subseteq{\mathcal A}\widehat{\otimes} {\mathcal A}$
is a bounded approximate diagonal for $\mathcal A$ if  
$$\lim_{\alpha}(a.\pi_{{\mathcal A}}(m_{\alpha})-{a})=0,\;\;\;and \;\;\;
\lim_{\alpha}(a.m_{\alpha}-m_{\alpha}.a)=0,\;\;\;\;(a\in \mathcal A).$$
Also according to \cite[Definition 6.2]{Pir}, if $\mathcal A$ is a complete
locally convex algebra, $\mathcal A$ has a locally bounded approximate
diagonal if for each zero neighborhood $U\subseteq \mathcal A$
there exists $C>0$ such that for each finite subset $F\subseteq \mathcal A$
there exists $M\in\overline{\Gamma(U\otimes U)}$ such that 
$a.M-M.a\in \overline{\Gamma(U\otimes U)}$ and $a.\pi_{\mathcal A}(M)-a\in U$
for all $a\in F$, where $\overline{\Gamma(U\otimes U)}$ denotes the closure 
of the absolutely convex hull of the set 
$$U\otimes U=\{a\otimes b:\;\; a,b\in U\}.$$
According to these definitions we can generalize the  definitions as follows.

\begin{definition}\label{D5} 
Let $(\mathcal A, p_n)$ be a  Fr$\acute{e}$chet algebra and   $I$ be a 
closed ideal of $\mathcal A$. We say that $\mathcal A$ has
 a bounded approximate diagonal modulo $I$, if there exists a bounded 
net $(m_{\alpha})_{\alpha}\subseteq  \frac{\mathcal A}{I}\widehat{\otimes}  \frac{\mathcal A}{I}$ such that 
$$\lim_{\alpha}(a.\pi_{\frac{\mathcal A}{I}}(m_{\alpha})-\tilde{a})=0 ,\;\;\;\;(a\in\mathcal A,\tilde{a}=a+I), $$
$$\lim_{\alpha}(a.m_{\alpha}-m_{\alpha}.a)=0\;\;\;\;(a\in \mathcal A\setminus  I).$$
Also we say that $\mathcal A$ has a locally bounded approximate diagonal 
modulo $I$, if there exists a family   $\{C_{n} :\;\;\; n\in\mathbb{N}\}$
of positive real numbers such that for each finite set $F\subseteq \mathcal A\setminus I$,
 each  $n\in \mathbb{N}$, and each $\varepsilon>0$ 
there exists $M\in (\frac{\mathcal A}{I}\widehat{\otimes}  \frac{\mathcal A}{I}, \hat{r}_n)$  such that
$$\hat{r}_n(M)\leq C_{n}\;\;\;\;and\;\;\;\hat{r}_n(a.M-M.a)<\varepsilon,\;\;\;(a\in F),$$
$$\hat{p}_n(a.\pi_{\frac{\mathcal A}{I}}(M)-\tilde{a})<\varepsilon,\;\;\;\;(a\in\mathcal A, \;\tilde{a}=a+I). $$
\end{definition}

In what follows, we restrict ourselves to the relation between
amenability modulo an ideal and the existence of a locally bounded 
approximate diagonal modulo an ideal for a  Fr$\acute{e}$chet algebra, 
it is to be noted that the  role played by reduced inverse limit is  essential.
\begin{theorem}\label{t5}
Let  $\varphi :(\mathcal A,p_n)\to (\mathcal B,q_m)$ be a continuous homomorphism
of  Fr$\acute{e}$chet algebras with dense range and let 
$I$ be a closed ideal of $\mathcal A$ and $J$ be a closed ideal of $\mathcal B$
such that $\varphi(I)\subseteq J$.  Suppose that
$\mathcal{A}$ has a  locally bounded approximate  diagonal modulo $I$.
 Then $\mathcal B$ has a locally bounded approximate diagonal  modulo $J$.
\end{theorem}

\begin{proof}
Let $q_{\mathcal B}:(\mathcal B,q_m)\to (\frac{\mathcal B}{J},\hat{q}_m)$ be 
the quotient map. So $q_{\mathcal B}\circ \varphi:\mathcal A\to\frac{\mathcal B}{J}$
is a continuous homomorphism with dense range. On the other hand $\varphi(I)\subseteq J$, 
therefore $I\subseteq \ker q_{\mathcal B}\circ \varphi$. Thus the function 
$\bar{\varphi}:(\frac{\mathcal A}{I},\hat{p}_n)\to(\frac{\mathcal B}{J},\hat{q}_m)$ defined by 
$\bar{\varphi}(a+I)= q_{\mathcal B}\circ \varphi(a)=\varphi(a)+J$ is well defined 
and it is  a continuous homomorphism. We can define
$\bar{\varphi}{\otimes}\bar{\varphi}:(\frac{\mathcal A}{I}\widehat{\otimes }\frac{\mathcal A}{I},\hat{r}_n)
\to (\frac{\mathcal B}{J}\widehat{\otimes}\frac{\mathcal B}{J},\hat{S}_m)$ by
\begin{eqnarray*}
\bar{\varphi}\otimes\bar{\varphi}((a+I)\otimes (b+I))&=&
\bar{\varphi}(a+I)\otimes\bar{\varphi}(b+I)
\\
&=&({\varphi}(a)+J)\otimes ({\varphi}(b)+J).
\end{eqnarray*}
Since $\bar{\varphi}\otimes\bar{\varphi}$ is a continuous map, for each $m\in\mathbb{N}$ 
there exist $\alpha_m>0$ and $n_1\in\mathbb{N}$ such that 
$$\hat{S}_m(\bar{\varphi}\otimes\bar{\varphi}(M))\leq \alpha_m \hat{r}_{n_1}(M), \;\;\;(M\in \frac{\mathcal A}{I}\widehat{\otimes }\frac{\mathcal A}{I}).$$
Without loss of generality we can assume that  $\alpha_m>1$.
On the other hand by the joint continuity of the module actions 
$$(\frac{\mathcal B}{J}\widehat{\otimes}\frac{\mathcal B}{J})\times {\mathcal B}\to\frac{\mathcal B}{J}\widehat{\otimes}\frac{\mathcal B}{J},\;\;\;(M^{\prime},b)\mapsto M^{\prime}.b,\;\;\;(M^{\prime}\in  \frac{\mathcal B}{J}\widehat{\otimes}\frac{\mathcal B}{J} , b\in {\mathcal B}),$$
and 
$${\mathcal B}\times(\frac{\mathcal B}{J}\widehat{\otimes}\frac{\mathcal B}{J})\to\frac{\mathcal B}{J}\widehat{\otimes}\frac{\mathcal B}{J},\;\;\;(b,M^{\prime})\mapsto b.M^{\prime},\;\;\;(M^{\prime}\in  \frac{\mathcal B}{J}\widehat{\otimes}\frac{\mathcal B}{J} , b\in {\mathcal B}),$$
and using \cite[Paga 7]{Pir} for each $m\in \mathbb{N}$, there exists  $m_0\in\mathbb{N}$ such that 
$$\hat{S}_m( M^{\prime}.b)\leq  \hat{S}_{m}( M^{\prime}) {q}_{m_0}(b), \;\;\;(M^{\prime}\in  \frac{\mathcal B}{J}\hat{\otimes}\frac{\mathcal B}{J} , b\in {\mathcal B}), $$
and 
 $$\hat{S}_m(b. M^{\prime})\leq  \hat{S}_{m}( M^{\prime}) {q}_{m_0}(b), \;\;\;(M^{\prime}\in  \frac{\mathcal B}{J}\widehat{\otimes}\frac{\mathcal B}{J} , b\in {\mathcal B}).$$
Also by \cite{Tay}, $\frac{\mathcal B}{J}$ is a  Fr$\acute{e}$chet  $\mathcal B$-bimodule by the module actions defined by 
$$\mathcal B\times \frac{\mathcal B}{J}\to \frac{\mathcal B}{J},\;\;\; (b, c+J)\mapsto bc+J, \;\;\;(b,c\in\mathcal B),$$
$$ \frac{\mathcal B}{J}\times\mathcal B\to \frac{\mathcal B}{J},\;\;\; ( c+J,b)\mapsto cb+J, \;\;\;(b,c\in\mathcal B),$$
and these two actions are jointly continuous. So by \cite[Page 7]{Pir}
for each $m\in\mathbb{N}$ there exists $m_2\in\mathbb{N}$ such that 
$$\hat{q}_m(bc+J)\leq q_{m_2}(b)\hat{q}_m(c+J), \;\;\; (b,c\in \mathcal B),$$
and 
$$\hat{q}_m(cb+J)\leq q_{m_2}(b)\hat{q}_m(c+J), \;\;\; (b,c\in \mathcal B).$$
Also the map $\bar{\varphi}:\frac{\mathcal A}{I}\to \frac{\mathcal B}{J}$ is continuous. So 
for each $m\in\mathbb{N}$ there exist $n_0\in\mathbb{N}$ and $k_m>0$  such that 
$$\hat{q}_m(\bar{\varphi}(a+I))\leq k_m \hat{p}_{n_0}(a+I), \;\;\; (a\in\mathcal A).$$
Without loss of generality, we may assume   that, $k_m>1$.
Let $\{C_n:\;\;n\in\mathbb{N}\}$ be a family of positive real numbers as in Definition \ref{D5}. 
Without loss of generality we may assume   that $C_n>1$ for each $n\in\mathbb{N}$.
Given a finite set $F^{\prime}\subseteq \mathcal B\setminus J$, $m\in\mathbb{N}$ and 
each $\varepsilon>0$, find a finite set $F\subseteq \mathcal A\setminus I$ such that 
$q_m(\varphi(a)-b)<\frac{\varepsilon}{3 \alpha_m C_{n_1}}$ for all $a\in F$ and $b\in F^{\prime}$.
Since $\mathcal A$
has  a locally bounded approximate diagonal modulo $I$, so 
  there exists $M\in\frac{\mathcal A}{I}\widehat{\otimes}\frac{\mathcal A}{I}$ such that
$$\hat{r}_n(M)\leq C_n  \;\;\; and \;\;\;\hat{r}_n(a.M-M.a)<\frac{\varepsilon}{3\alpha_m}\;\;\;(a\in F).$$
Also $$\hat{p}_n(a.\pi_{\frac{\mathcal A}{I}}(M) - \tilde{a})<\frac{\varepsilon}{3k_m}, \;\;\;(a\in\mathcal A, \;\; \tilde{a}=a+I).$$
Put $M^{\prime}=\bar{\varphi}\otimes \bar{\varphi}(M)\in \frac{\mathcal B}{J}\widehat{\otimes}\frac{\mathcal B}{J}.$
Therefore
 $$\hat{S}_m(M^{\prime})=\hat{S}_m(\bar{\varphi}\otimes \bar{\varphi}(M))\leq \alpha_m\hat{r}_{n_1}(M) \leq \alpha_m C_{n_1}.$$
Now take $b\in F^{\prime}$ and choose $ a\in F$ satisfying $q_m(\varphi(a)-b)<\frac{\varepsilon}{3 \alpha_m C_{n_1}}.$
Thus 
\begin{eqnarray*}
\hat{S}_m(b.M^{\prime}-M^{\prime}.b)&=&\hat{S}_m(b.\bar{\varphi}\otimes \bar{\varphi}(M)-\bar{\varphi}\otimes \bar{\varphi}(M).b)
\\
&=&\hat{S}_m(\bar{\varphi}\otimes \bar{\varphi}(a.M-M.a)+(b-\varphi(a)).M^{\prime}-M^{\prime}.(b-\varphi(a))
\\
&\leq & \hat{S}_m(\bar{\varphi}\otimes \bar{\varphi}(a.M-M.a))+\hat{S}_m((b-\varphi(a)).M^{\prime})+\hat{S}_m(M^{\prime}.(b-\varphi(a)))\\
&\leq & \alpha_m \hat{r}_{n_1}(a.M-M.a)+ 2 q_{m_0}(b-\varphi(a)) \hat{S}_{m}(M^{\prime})\\
&\leq & \alpha_m \hat{r}_{n_1}(a.M-M.a)+ 2 q_{m_0}(b-\varphi(a)) \hat{S}_{m_0}(M^{\prime})\\
&< & \alpha_m \frac{\varepsilon}{3\alpha_m}+ 2\alpha_{m_0} C_{n_1}\frac{\varepsilon}{3\alpha_{m_0} C_{n_1}}\\
&=&\varepsilon.
\end{eqnarray*}
In the above relations we can assume that $m_0> m$. In fact if $m_0\leq m$, we have 
 \begin{eqnarray*}
\hat{S}_m(b.M^{\prime}-M^{\prime}.b)&=&\hat{S}_m(b.\bar{\varphi}\otimes \bar{\varphi}(M)-\bar{\varphi}\otimes \bar{\varphi}(M).b)
\\
&=&\hat{S}_m(\bar{\varphi}\otimes \bar{\varphi}(a.M-M.a)+(b-\varphi(a)).M^{\prime}-M^{\prime}.(b-\varphi(a))
\\
&\leq & \hat{S}_m(\bar{\varphi}\otimes \bar{\varphi}(a.M-M.a))+\hat{S}_m((b-\varphi(a)).M^{\prime})+\hat{S}_m(M^{\prime}.(b-\varphi(a)))\\
&\leq & \alpha_m \hat{r}_{n_1}(a.M-M.a)+ 2 q_{m_0}(b-\varphi(a)) \hat{S}_{m}(M^{\prime})\\
&\leq & \alpha_m \hat{r}_{n_1}(a.M-M.a)+ 2 q_{m}(b-\varphi(a)) \hat{S}_{m}(M^{\prime})\\
&\leq & \alpha_m \frac{\varepsilon}{3\alpha_m}+ 2\alpha_{m} C_{n_1}\frac{\varepsilon}{3\alpha_{m} C_{n_1}}\\
&=&\varepsilon.
\end{eqnarray*}
Furthermore $\pi_{\frac{\mathcal B}{J}}$ is a $\mathcal B$-bimodule
homomorphism and the quotient map is seminorms decreasing, so
\begin{eqnarray*}
\hat{q}_m(b\cdot\pi_{\frac{\mathcal B}{J}}(M^{\prime})-\tilde{b}) &\leq& \hat{q}_m((b-{\varphi}(a))\cdot\pi_{\frac{\mathcal B}{J}}(M^{\prime}))+
\hat{q}_m(\bar{\varphi}(a\cdot\pi_{\frac{\mathcal A}{I}}(M)-\tilde{a}))+\hat{q}_{m}(\bar{\varphi}(\tilde{a})-\tilde{b})
\\
&=&\hat{q}_m(\pi_{\frac{\mathcal B}{J}}((b-\varphi(a)).M^{\prime}))+
\hat{q}_m(\bar{\varphi}(a.\pi_{\frac{\mathcal A}{I}}(M)-\tilde{a}))+\hat{q}_m(\bar{\varphi}(\tilde{a})-\tilde{b})\\
&\leq& \hat{S}_m((b-{\varphi}(a)).M^{\prime})+
\hat{q}_m(\bar{\varphi}(a.\pi_{\frac{\mathcal A}{I}}(M)-\tilde{a}))+\hat{q}_m(\bar{\varphi}(\tilde{a})-\tilde{b})\\
&\leq &\hat{S}_m(M^{\prime})q_{m_0}(b-\varphi(a))+k_m \hat{p}_{n_0}(a.\pi_{\frac{\mathcal A}{I}}(M)-\tilde{a})+q_m(\varphi(a)-b)
\\
&\leq & \hat{S}_m(M^{\prime})q_{m_0}(b-\varphi(a))+ k_m \frac{\varepsilon}{3 k_m}+\frac{\varepsilon}{3 \alpha_{m} C_{n_1}} \\
&\leq & \hat{S}_{m_0}(M^{\prime})q_{m_0}(b-\varphi(a))+ k_m \frac{\varepsilon}{3 k_m}+\frac{\varepsilon}{3 \alpha_{m} C_{n_1}} \\
&\leq & \alpha_{m_0} C_{n_1} \frac{\varepsilon}{3 \alpha_{m_0} C_{n_1}}+ k_m \frac{\varepsilon}{3 k_m}+\frac{\varepsilon}{3 \alpha_{m} C_{n_1}}  \\
&=&\varepsilon,
\end{eqnarray*}
for each $b\in F^{'}$ and $\tilde{b}=b+J$. Note that in the above relations we
can assume that, $m_0>m$.
Henceforth $\mathcal B$ has a locally bounded approximate diagonal modulo $J$.
\end{proof}

\begin{proposition}\label{p10}
Let $(\mathcal A,p_n)$ be a  Fr$\acute{e}$chet algebra and $I$ be a closed ideal of
 $\mathcal A$. $\mathcal A$ has a bounded approximate diagonal modulo $I$ 
 if and only if there exists a bounded set $B\subseteq \frac{\mathcal A}{I}\widehat{\otimes}  \frac{\mathcal A}{I}$
such that   for each finite set $F\subseteq \mathcal A\setminus I$,
 each  $n\in \mathbb{N}$, and each $\varepsilon>0$ 
there exists $M\in (\frac{\mathcal A}{I}\widehat{\otimes}  \frac{\mathcal A}{I}, \hat{r}_n)$  such that
$$\hat{r}_n(a.M-M.a)<\varepsilon,\;\;\;(a\in F),$$
$$\hat{p}_n(a.\pi_{\frac{\mathcal A}{I}}(M)-\tilde{a})<\varepsilon,\;\;\;\;(a\in\mathcal A, \;\tilde{a}=a+I). $$
\end{proposition}

\begin{proof}
First suppose that $(m_{\alpha})_{\alpha}\subseteq  \frac{\mathcal A}{I}\widehat{\otimes}  \frac{\mathcal A}{I}$
is a bounded approximate diagonal modulo $I$ for $\mathcal A$.
So for   each  $n\in \mathbb{N}$, and each $\varepsilon>0$ 
there exists $\alpha_0$ such that for each $\alpha\geq \alpha_0$ we have 
$$\hat{p}_n(a.\pi_{\frac{\mathcal A}{I}}(m_{\alpha})-\tilde{a})<\varepsilon,\;\;\;\;(a\in\mathcal A,\tilde{a}=a+I), $$
$$\hat{r}_n(a.m_{\alpha}-m_{\alpha}.a)<\varepsilon\;\;\;\;(a\in \mathcal A\setminus  I).$$
Now  set $B=(m_{\alpha})_{\alpha}\subseteq  \frac{\mathcal A}{I}\widehat{\otimes}  \frac{\mathcal A}{I}$.
Let $n\in \mathbb{N}$, $\varepsilon>0$  and $F\subseteq \mathcal A\setminus I$,
be a finite set. So there exists $M=m_{\alpha}$ which ${\alpha\geq \alpha_0}$
such that 
$$\hat{r}_n(a.M-M.a)<\varepsilon,\;\;\;(a\in F),$$
$$\hat{p}_n(a.\pi_{\frac{\mathcal A}{I}}(M)-\tilde{a})<\varepsilon,\;\;\;\;(a\in\mathcal A, \;\tilde{a}=a+I). $$
Conversely, suppose that there exists a bounded set $B\subseteq \frac{\mathcal A}{I}\widehat{\otimes}  \frac{\mathcal A}{I}$
such that   for each finite set $F\subseteq \mathcal A\setminus I$,
 each  $n\in \mathbb{N}$, and each $\varepsilon>0$ 
there exists $M\in (\frac{\mathcal A}{I}\widehat{\otimes}  \frac{\mathcal A}{I}, \hat{r}_n)$  such that
$$\hat{r}_n(a.M-M.a)<\varepsilon,\;\;\;(a\in F),$$
$$\hat{p}_n(a.\pi_{\frac{\mathcal A}{I}}(M)-\tilde{a})<\varepsilon,\;\;\;\;(a\in\mathcal A, \;\tilde{a}=a+I). $$
Now take 
$$S=\{(n,\varepsilon, F):\;\;n\in\mathbb{N},\;\;\varepsilon>0\;\;\; and\;\;\; F\subseteq \mathcal A\setminus I \;\;  is \;\;a\;\; finite\;\; set  \}.$$
So $S$ is a directed set as follows:
$$(n_1,\varepsilon_1, F_1)\leq (n_2,\varepsilon_2, F_2)\;\;\Leftrightarrow\;\;n_1\leq n_2\;\;\; and\;\;\varepsilon_2\leq\varepsilon_1\;\;\; and \;\;F_1\subseteq  F_2.$$
Therefore there exists a bounded net 
 $(m_{\alpha})\subseteq B\subseteq \frac{\mathcal A}{I}\widehat{\otimes}  \frac{\mathcal A}{I}$.
Furthermore let $k\in\mathbb{N}$ and $\varepsilon>0$ and the finite set $F\subseteq \mathcal A\setminus I$
be arbitrary. Therefore for each $(n,\delta, C)\geq(k,\varepsilon, F)$
we have
$$\hat{p}_{k}(a.\pi_{\frac{\mathcal A}{I}}(m_{\alpha})-\tilde{a})\leq \hat{p}_n(a.\pi_{\frac{\mathcal A}{I}}(m_{\alpha})-\tilde{a})<\delta<\varepsilon,\;\;\;\;(a\in\mathcal A, \tilde{a}=a+I), $$
$$\hat{r}_{k}(a.m_{\alpha}-m_{\alpha}.a)\leq \hat{r}_{n}(a.m_{\alpha}-m_{\alpha}.a)<\delta<\varepsilon\;\;\;\;(a\in \mathcal A\setminus  I).$$
So $\mathcal A$ has a bounded approximate diagonal modulo $I$.
\end{proof}

An immediate consequence of Definition \ref{D5} and Proposition \ref{p10} is the
following.
\begin{remark}\label{r5}
For a Banach algebra  the notions of locally 
bounded and bounded approximate diagonal modulo an ideal  are equivalent.
\end{remark}
Let $(\mathcal A,p_n)$ be a  Fr$\acute{e}$chet algebra and $I$ be a closed ideal of
 $\mathcal A$. As we have already mentioned $(\frac{\mathcal A}{I}, \hat{p}_n)$ and 
 $(\frac{\mathcal A}{I}\widehat{\otimes}\frac{\mathcal A}{I}, \hat{r}_n)$ are   Fr$\acute{e}$chet algebras
and for each $n\in\mathbb{N}$, $(\mathcal A_{n} , \|\cdot \|_n)$
 is a Banach algebra where $\|\varphi_n(a){\|}_n=\|a+\ker p_n\|_n=p_n(a)$. Also
  $(\frac{\mathcal A_{n}}{\overline{I_n}}, \|\cdot\hat{\|}_{n})$ is a Banach algebra, 
 where  $$\|\varphi_n(a)+\overline{I_n}\hat{\|}_{n}=\inf\{\|\varphi_n(a)+b\|_n:\;\;\; b\in \overline{I_n}\},$$
and $(\frac{\mathcal A_{n}}{\overline{I_n}}\widehat\otimes\frac{\mathcal A_{n}}{\overline{I_n}}, |\|\cdot|\|_n)$ is a 
Banach algebra, where $|\|\cdot|\|_n$, is the projective tensor norm of Banach algebras 
  $(\frac{\mathcal A_{n}}{\overline{I_n}}, \|\cdot\hat{\|}_{n})$.
 On the other hand 
\begin{eqnarray*}
\hat{p}_n(a+I)&=&
\inf\{p_n(a+y):\;\;\;y\in I\}\\
&=& \inf\{\|\varphi_n(a+y)\|_n:\;\;\;y\in I\}\\
&=& \inf\{\|\varphi_n(a)+\varphi_n(y)\|_n:\;\;\;y\in I\}\\
&=& \inf\{\|\varphi_n(a)+z\|_n:\;\;\;z\in { I_n}\}\\
&\geq & \inf\{\|\varphi_n(a)+z\|_n:\;\;\;z\in\overline{ I_n}\}\\
&=&\|\varphi_n(a)+\overline{ I_n}\hat{\|}_n.
\end{eqnarray*}
Also if $z\in\overline{ I_n}$, then there is a sequence $(z_{m,n})_{m\in\mathbb{N}}\in I_n$
such that $z=\lim_m z_{m,n}$. Now we have 
\begin{eqnarray*}
\|\varphi_n(a)+ z\hat{\|}_n&=&
\lim_m \|\varphi_n(a)+ z_{m,n}\hat{\|}_n \\
&\geq & \inf_m \|\varphi_n(a)+ z_{m,n}\hat{\|}_n\\
&\geq & \inf\{\|\varphi_n(a)+y_n\|_n:\;\;\;y_n\in I_n\}\\
&= &  \inf\{\|\varphi_n(a)+\varphi_n(y)\|_n:\;\;\;y\in I\}\\
&= &  \inf\{\|\varphi_n(a+y)\|_n:\;\;\;y\in I\}\\
&=& \inf\{p_n(a+y):\;\;\;y\in I\}\\
&=&\hat{p}_n(a+I).
\end{eqnarray*}
Therefore $\hat{p}_n(a+I)\leq \|\varphi_n(a)+ \overline{I_n}\hat{\|}_n$,
and so $\hat{p}_n(a+I)= \|\varphi_n(a)+ \overline{I_n}\hat{\|}_n$. 

Also
\begin{eqnarray*}
\hat{r}_n(M)&=&\inf\{\sum_{i=1}^{\infty}\hat{p}_n(a_i+I)\hat{p}_n(b_i+I):\;\;M=\sum_{i=1}^{\infty}(a_i+I)\otimes(b_i+I),\;\;
M\in \frac{\mathcal A}{I}\widehat{\otimes}\frac{\mathcal A}{I}\}\\
&=&\inf\{\sum_{i=1}^{\infty}\|\varphi_n(a_i)+\overline{ I_n}\hat{\|}_n\|\varphi_n(b_i)+\overline{ I_n}\hat{\|}_n:\;\;M=\sum_{i=1}^{\infty}(a_i+I)\otimes(b_i+I),\;\;
M\in \frac{\mathcal A}{I}\widehat{\otimes}\frac{\mathcal A}{I}\}\\
&=&|\|\bar{\varphi}_n\otimes\bar{\varphi}_n(M)\||_n,
\end{eqnarray*}
Where $\varphi_n$ stands for the canonical map from $\mathcal A$ into 
$\mathcal A_n$ and $\bar{\varphi}_n$ is defined from  $\frac{\mathcal A}{I}$
into $\frac{\mathcal A_n}{\overline{ I_n}}$.
\begin{proposition}\label{p8}
Let $(\mathcal A,p_n)$ be a  Fr$\acute{e}$chet algebra and $I$ be a closed ideal of
 $\mathcal A$. Then the following conditions are equivalent:
\begin{enumerate}
\item[(i)]   $\mathcal A$  has a locally bounded approximate diagonal modulo $I$.
 \item[(ii)]
For each Banach algebra $\mathcal B$ and each continuous homomorphism 
$\varphi:\mathcal A\to\mathcal B$ with dense range,  the Banach algebra $\mathcal B$ has a  bounded 
approximate diagonal modulo $J$, where $J$ is a closed ideal of $\mathcal B$
 such that $\varphi(I)\subseteq J$.
\end{enumerate}
\end{proposition}

\begin{proof}
$(i)\Rightarrow (ii)$. It is straightforward from Theorem \ref{t5} and Remark \ref{r5}.\\
$(ii)\Rightarrow (i)$. Let $n\in\mathbb{N}$ be arbitrary and $\mathcal A=\underleftarrow{\lim}\mathcal A_n$
be an Arens-Michael decomposition of $\mathcal A$ and $I=\underleftarrow{\lim}\overline{I_n}$ be an 
Arens-Michael decomposition of $I$.  Since $\varphi_n:\mathcal A\to\mathcal A_n$
is a continuous homomorphism with dense range, by assumption the Banach algebra $\mathcal A_n$
has a bounded approximate diagonal modulo $\overline{I_n}$. 
Therefore by Proposition \ref{p10}, there exists  $C_n>0$ such that for each finite set 
$F_n\subseteq {\mathcal A_{n}}\setminus{ \overline{I_n}}$ and each
$\varepsilon>0$ there exists  
$M^{'}\in\frac{\mathcal A_{n}}{ \overline{I_n}}\widehat{\otimes}\frac{\mathcal A_{n}}{ \overline{I_n}}$
for which
$$|\| M^{'}\||_n\leq C_n$$
and
$$|\|b. M^{'}-M^{'}.b\||_n<\frac{\varepsilon}{3}\;\;\;, (b\in F_n)$$
and
$$\|b. \pi_{\frac{\mathcal A_{n}}{ \overline{I_n}}}(M^{'})-\tilde{b}\hat{\|}_n<\frac{\varepsilon}{3},\;\;\; (b\in {\mathcal A}_n, \tilde{b}=b+ \overline{I_n}).$$
Suppose that $q_{\mathcal A_{n}}:\mathcal A_{n}\to {\frac{\mathcal A_{n}}{ \overline{I_n}}}$ be
the quotient map. Then $q_{\mathcal A_{n}}\circ \varphi_n:\mathcal A\to {\frac{\mathcal A_{n}}{ \overline{I_n}}}$ 
defined by $q_{\mathcal A_{n}}\circ \varphi_n (a)=\varphi_n(a)+{ \overline{I_n}}$ is a continuous homomorphism
with dense range and $I\subseteq \ker(q_{\mathcal A_{n}}\circ \varphi_n)$.
Therefore we can define a continuous homomorphism with dense range 
$\bar{\varphi}_n:\frac{\mathcal A}{I}\to \frac{\mathcal A_{n}}{ \overline{I_n}}$ 
such that $\bar{\varphi}_n(a+I)={\varphi}_n(a)+\overline{I_n}$ and so
$\bar{\varphi}_n\otimes\bar{\varphi}_n:\frac{\mathcal A}{I}\widehat{\otimes}\frac{\mathcal A}{I}\to \frac{\mathcal A_{n}}{ \overline{I_n}}\widehat{\otimes} \frac{\mathcal A_{n}}{ \overline{I_n}}$
is a continuous homomorphism with dense range.
If $F=\{a_1, a_2, ... , a_m\}\subseteq  \mathcal A\setminus I$  is a finite set, then there exists $n_0\in\mathbb{N}$,
such that $a_n^i\notin\overline{I_n}$ 
for each $n\geq n_0$, where $a_i=(a_n^i)$ and $1\leq i\leq m$. In fact if $a_i\notin I$, then there exists 
$n_i\in\mathbb{N}$, such that $a_{n_i}^i\notin  \overline{I_{n_i}}$. Also the maps 
$\varphi_{mn}:\mathcal A_m\to \mathcal A_n$ defined by $\varphi_{mn}(a_m)=a_n$, 
for each $m\geq n$. So $a_n^i\notin \overline{I_{n}}$, for each $n\geq n_i$. Now put 
$$n_0=\max\{ n_i:\;\; 1\leq i\leq m\}.$$
Thus  $a_n\notin\overline{I_n}$, for each $n\geq n_0$.
In the sequel take a
 finite set $F\subseteq \mathcal A\setminus I$, $\varepsilon>0$ 
and  $k\in\mathbb{N}$.  Also set 
$$n\geq\max \{n_0,k\}.$$
Because of the above argument we can assume that
 $F^{\prime}=\{\varphi_n(a):\,\, a\in F\}\subseteq \mathcal A_n\setminus \overline{I_n}$
  is a finite set.
Put $$C=\max\{\|\varphi_n(a)\|_n:\;\;a\in F\}.$$
Since  $\bar{\varphi}_n\otimes\bar{\varphi}_n $
 is a continuous homomorphism with dense range, so there exists $M\in \frac{\mathcal A}{I}\hat{\otimes}\frac{\mathcal A}{I}$
such that  
$$\||\bar{\varphi}_n\otimes\bar{\varphi}_n(M)-M^{\prime}\||_n<\frac{\varepsilon}{3C}.$$
Hence 
\begin{eqnarray*}
\hat{r}_k(a.M-M.a)&\leq&\hat{r}_n(a.M-M.a)\\&=&
\||\bar{\varphi}_n\otimes\bar{\varphi}_n(a.M-M.a)\||_n\\
&=& |\|\varphi_n(a).\bar{\varphi}_n\otimes\bar{\varphi}_n(M)-\bar{\varphi}_n\otimes\bar{\varphi}_n(M).\varphi_n(a)\||_n\\
&\leq&| \|\varphi_n(a).\bar{\varphi}_n\otimes\bar{\varphi}_n(M)-\varphi_n(a).M^{\prime}\||_n\\
&+&|\|M^{\prime}.\varphi_n(a)- \varphi_n(a).M^{\prime}|\|_n\\
&+&|\|M^{\prime}.\varphi_n(a) -\bar{\varphi}_n\otimes\bar{\varphi}_n(M).\varphi_n(a)\||_n\\
&\leq& \|\varphi_n(a)\|_n\||\bar{\varphi}_n\otimes\bar{\varphi}_n(M)-\varphi_n(a).M^{\prime}\||_n\\
&+&|\|M^{\prime}.\varphi_n(a)- \varphi_n(a).M^{\prime}|\|_n\\
&+&\|\varphi_n(a)\|_n\||\bar{\varphi}_n\otimes\bar{\varphi}_n(M)-\varphi_n(a).M^{\prime}\||_n\\
&<& 2C\frac{\varepsilon}{3C}+\frac{\varepsilon}{3}=\varepsilon.
\end{eqnarray*}
Furthermore 
\begin{eqnarray*}
\hat{p}_k(a.\pi_{\frac{\mathcal A}{I}}(M)-\tilde{a})&\leq&\hat{p}_n(a.\pi_{\frac{\mathcal A}{I}}(M)-\tilde{a})\\&=&
\|\bar{\varphi}_n(a.\pi_{\frac{\mathcal A}{I}}(M)-\tilde{a})\hat{\|}_n\\
&=&\|\varphi_n(a)\pi_{\frac{\mathcal A_n}{\overline{I_n}}}(\bar{\varphi}_n\otimes\bar{\varphi}_n(M))-\bar{\varphi}_n(\tilde{a})\hat{\|}_n\\
&=&\|\varphi_n(a)\pi_{\frac{\mathcal A_n}{\overline{I_n}}}(\bar{\varphi}_n\otimes\bar{\varphi}_n(M))-\bar{\varphi}_n(a+I)\hat{\|}_n\\
&=&\|\varphi_n(a)\pi_{\frac{\mathcal A_n}{\overline{I_n}}}(\bar{\varphi}_n\otimes\bar{\varphi}_n(M))-(\varphi_n(a)+\overline{I_n})\hat{\|}_n\\
&=&\|\varphi_n(a)\pi_{\frac{\mathcal A_n}{\overline{I_n}}}(\bar{\varphi}_n\otimes\bar{\varphi}_n(M))-\widetilde{\varphi_n(a)}\hat{\|}_n\\
&\leq& \|\varphi_n(a)\|_n\|\pi_{\frac{\mathcal A_n}{\overline{I_n}}}(\bar{\varphi}_n\otimes\bar{\varphi}_n(M)-M^{\prime})\hat{\|}_n\\
&+&  \|\varphi_n(a)\pi_{\frac{\mathcal A_n}{\overline{I_n}}}(M^{\prime})-\widetilde{\varphi_n(a)})\hat{\|}_n\\
&\leq& \|\varphi_n(a)\|_n|\|\bar{\varphi}_n\otimes\bar{\varphi}_n(M)-M^{\prime}|\|_n\\
&+&\|\varphi_n(a)\pi_{\frac{\mathcal A_n}{\overline{I_n}}}(M^{\prime})-\widetilde{\varphi_n(a)})\hat{\|}_n\\
&\leq& C |\|\bar{\varphi}_n\otimes\bar{\varphi}_n(M)-M^{\prime}|\|_n+\frac{\varepsilon}{3}\\
&\leq& C \frac{\varepsilon}{3C}+\frac{\varepsilon}{3}\\
&<&\varepsilon.
\end{eqnarray*}
Which in view  of  Definition \ref{D5}, this completes the proof.
\end{proof}

\begin{corollary}\label{Cor6}
Let $(\mathcal A,p_n)$ be a  Fr$\acute{e}$chet algebra and $I$ be a closed ideal of
 $\mathcal A$. Suppose that $\mathcal A=\underleftarrow{\lim}\mathcal A_n$
is an Arens-Michael decomposition of $\mathcal A$ and $I=\underleftarrow{\lim}\overline{I_n}$ is an 
Arens-Michael decomposition of $I$. Then $\mathcal A$ 
has a locally bounded approximate diagonal  modulo $I$ if and only if
each  $\mathcal A_n$ has a  bounded approximate diagonal modulo $\overline{I_n}$. 
\end{corollary}

\begin{proof}
Since $\varphi_n:\mathcal A\to \mathcal A_n$ is a continuous homomorphism
with dense range, Theorem \ref{t5} and Remark \ref{r5}, imply that if $\mathcal A$ has a locally
 bounded approximate diagonal modulo $I$, then $\mathcal A_n$ has a
 bounded approximate  diagonal modulo $\overline{I_n}$, for each $n\in\mathbb{N}$.
\\Conversely, suppose that each $\mathcal A_n$ has a bounded approximate diagonal
modulo $\overline{I_n}$ and  $\Phi:\mathcal A\to \mathcal B$ is a continuous 
homomorphism for some Banach algebra $\mathcal B$, 
 with dense range. Since $\mathcal B$ is a Banach algebra and $\Phi$ 
 is a continuous homomorphism, there exists a continuous homomorphism
 with dense range $\varPhi_{n_0}:\mathcal A_{n_0}\to \mathcal B$
for some $n_0\in \mathbb{N}$.
 On the other hand if $J$ is a closed ideal of $\mathcal B$
such that $\varPhi(I)\subseteq J$, then $\varPhi_{n_0}(I_{n_0})=\varPhi(I)\subseteq J$.
Therefore by Theorem \ref{t5},  $\mathcal B$ has a bounded approximate diagonal
modulo $J$ such that $\Phi(I)\subseteq J$. Thus by using Proposition \ref{p8},
$\mathcal A$ has a locally bounded approximate diagonal modulo $I$.
\end{proof}

We now state the main result of this section which 
 is an extention of corresponding result for Banach algebras.
\begin{theorem}\label{t13}
Let $(\mathcal A,p_n)$ be a  Fr$\acute{e}$chet algebra and $I$ be a closed ideal of
 $\mathcal A$. Then the following statements are equivalent;
 \begin{enumerate}
\item[(i)]   $\mathcal A$  is amenable modulo $I$.
 \item[(ii)]
 $\mathcal A$ has a locally bounded 
approximate diagonal modulo $I$.
\end{enumerate}
\end{theorem}

\begin{proof}
$(i)\Rightarrow (ii)$. Suppose that  $\mathcal A=\underleftarrow{\lim}\mathcal A_n$
is an Arens-Michael decomposition of $\mathcal A$ and $I=\underleftarrow{\lim}\overline{I_n}$ is an 
Arens-Michael decomposition of $I$. Since $\mathcal A$ is amenable modulo $I$,
by \cite[Theorem 2.3]{Rahnama}   each $\mathcal A_n$ is amenable modulo $\overline{I_n}$. 
So   $\mathcal A_n$ has a 
bounded approximate diagonal modulo $\overline{I_n}$ for each 
$n\in\mathbb{N}$ by \cite[Theorem 7]{Rahimi1}. Then    $\mathcal A$ has a locally bounded 
approximate diagonal modulo $I$, by Corollary \ref{Cor6}.
\\
$(ii)\Rightarrow (i)$. Let  $\mathcal A$ has a locally bounded 
approximate diagonal modulo $I$ and $\mathcal A=\underleftarrow{\lim}\mathcal A_n$
be an Arens-Michael decomposition of $\mathcal A$ and $I=\underleftarrow{\lim}\overline{I_n}$ be an 
Arens-Michael decomposition of $I$.  Therefore by Corollary \ref{Cor6},
 $\mathcal A_n$ has a 
bounded approximate diagonal modulo $\overline{I_n}$ for each 
$n\in\mathbb{N}$. So  each $\mathcal A_n$
is amenable modulo $\overline{I_n}$, by \cite[Theorem 7]{Rahimi1}. Thus 
 $\mathcal A$  is amenable modulo $I$, by \cite[Theorem 2.3]{Rahnama} .
\end{proof}

\begin{corollary}\label{cor10}
Let $(\mathcal A,p_n)$ be a  Fr$\acute{e}$chet algebra and $I$ be a closed ideal of
 $\mathcal A$. If $\mathcal A$ has a locally bounded approximate diagonal modulo $I$, then
 $\frac{\mathcal A}{I}$ has a locally bounded approximate diagonal. 
\end{corollary}
\begin{proof}
Since $q_{\mathcal A}:\mathcal A\to \frac{\mathcal A}{I}$ is a continuous homomorphism with dense range 
by Theorem \ref{t5}, $\frac{\mathcal A}{I}$ has a locally bounded approximate diagonal modulo 
$I=\{0_{\frac{\mathcal A}{I}}\}$. So it is clear that $\frac{\mathcal A}{I}$ is amenable. By 
applying \cite[Theorem 9.7]{Pir}, $\frac{\mathcal A}{I}$ has a locally bounded approximate diagonal.
\end{proof}

In the next propositions we give two hereditary properties of amenability 
modulo an ideal for Fr$\acute{e}$chet algebras.
\begin{proposition}
Let $(\mathcal A, p_n)$  be a  Fr$\acute{e}$chet algebra and $I$ be a closed ideal of
$\mathcal A$. Then $\mathcal A$ is amenable modulo $I$  if and only if $\mathcal A_{e}$ 
 is amenable modulo $I$.
\end{proposition}

\begin{proof}
Suppose that $\mathcal A=\underleftarrow{\lim}\mathcal A_n$
be an Arens-Michael decomposition of $\mathcal A$ and $I=\underleftarrow{\lim}\overline{I_n}$ is an 
Arens-Michael decomposition of $I$. By \cite[Theorem 2.3]{Rahnama},   $\mathcal A$ is
 amenable modulo $I$ if and only if for 
 each $n\in\mathbb{N}$, $\mathcal A_n$ is amenamble modulo $\overline{I_n}$. 
This is equivalent to     $(\mathcal A_n)_{e}$ is 
amenamble modulo $\overline{I_n}$, for each $n\in\mathbb{N}$ by  
  using \cite[Theorem 3.2]{Rahimi3}.
On the other hand by \cite[Proposition 3.11]{F},  $\mathcal A_e=\underleftarrow{\lim}(\mathcal A_n)_{e}$.
So $\mathcal A$ is amenable modulo $I$  if  and only if ${\mathcal A}_{e}$ 
 is amenable modulo $I$.
\end{proof}
\begin{proposition}
Let $(\mathcal A,p_n)$ be a  Fr$\acute{e}$chet algebra and $I$ be a closed ideal of
 $\mathcal A$ and $J$ be a closed ideal of $I$. If  $\mathcal A$  is amenable modulo $I$, then
 $\frac{\mathcal A}{J}$ is amenable modulo $\frac{I}{J}$. 
\end{proposition}
\begin{proof}
Let $\mathcal A=\underleftarrow{\lim}\mathcal A_n$
be an Arens-Michael decomposition of $\mathcal A$ and $I=\underleftarrow{\lim}\overline{I_n}$ be an 
Arens-Michael decomposition of $I$ and $J=\underleftarrow{\lim}\overline{J_n}$
be an Arens-Michael decomposition of $J$. Since  $\mathcal A$  is amenable modulo $I$,
by Theorem  \cite[Theorem 2.3]{Rahnama}, each $\mathcal A_n$
is amenable modulo $\overline{I_n}$. Therefore 
$\frac{\mathcal A_n}{\overline{J_n}}$, is amenable modulo $\frac{\overline{I_n}}{\overline{J_n}}$ 
for each $n\in\mathbb{N}$, by \cite[Theorem 3.8]{Rahimi3}.
So  $\frac{\mathcal A}{J}$ is amenable modulo $\frac{I}{J}$, by \cite[Theorem 2.3]{Rahnama}.  
\end{proof}

{\bf Acknowledgment.} This work was supported by Iran national science  foundation and we
would like to acknowledge it with much appreciation. Also we are thankful
to university of Isfahan in developing the project.

\vspace{9mm}

{\footnotesize \noindent
 S. Rahnama\\
  Department of Mathematics,
   University of Isfahan,
    Isfahan, Iran\\
     rahnama$_{-}600$@yahoo.com\\

\noindent
 A. Rejali\\
  Department of Mathematics,
   University of Isfahan,
    Isfahan, Iran\\
    rejali@sci.ui.ac.ir\\

\end{document}